\documentclass{amsart}
\usepackage{amsfonts,amssymb,amscd,amsmath,enumerate,verbatim,newlfont,calc}

\newtheorem{theorem}{Theorem}[section]
\newtheorem{lemma}[theorem]{Lemma}
\newtheorem{proposition}[theorem]{Proposition}
\newtheorem{corollary}[theorem]{Corollary}
\theoremstyle{definition}

\newtheorem{examples}[theorem]{Examples}

\begin{document}

\title [Algebraic Properties upon Graph Operations]{Vertex decomposability, shellability and Cohen-Macaulayness of graphs upon graph operations}

\author {Fahimeh Khosh-Ahang Ghasr}

\address{Fahimeh Khosh-Ahang, Department of Mathematics, School of Science, Ilam University, P.O.Box 69315-516, Ilam, Iran}
\email{fahime$_{-}$khosh@yahoo.com and f.khoshahang@ilam.ac.ir}

%\dedicatory{ }

\begin{abstract}
Throughout this work, the vertex decomposability and shellability of graphs formed from other graphs by various operations are investigated. Also among the other things, by using some graph operations, new classes of Cohen-Macaulay graphs from previous ones are presented.
% (disjoint union, join, rooted product, corona, cartesian product and lexicographic product).
%Conditions for the consequence of these operations of graphs to be vertex decomposable (shellable) are established based upon the factors and vice versa. 

\end{abstract}

\thanks{}

%		13H10   	Special types (Cohen-Macaulay, Gorenstein, Buchsbaum, etc.)
%		13D02
 %  	Syzygies, resolutions, complexes
%		05E40   	Combinatorial aspects of commutative algebra
%		16S36   	Ordinary and skew polynomial rings and semigroup rings

%		14M25   	Toric varieties, Newton polyhedra [See also 52B20]
%		13A02   	Graded rings
%		13F20   	Polynomial rings and ideals; rings of integer-valued polynomials
%		13A18   	Valuations and their generalizations
%		06A11   	Algebraic aspects of posets
%     05C10      Planar graphs; geometric and topological aspects of graph theory
%     05C76      Graph operations (line graphs, products, etc.)
%     05E45     Combinatorial aspects of simplicial complexes


\maketitle

\setcounter{tocdepth}{1}
%\tableofcontents

\section{Introduction}
Vertex decomposability and shellability are some topological combinatorial notions which are related to the
algebraic properties of the Stanley-Reisner ring of a simplicial complex. These concepts were first introduced in the pure case \cite{Mc, Provan+Billera, Rudin} and then extended to non-pure complexes \cite{BW, BW2}. 

Let $G=(V,E)$ be a simple graph. For each vertex $x$ in $G$, $N_G(x)$ denotes the set of vertices adjacent to $x$ and $N_G[x]=N_G(x)\cup\{x\}$. $N_G(x)$ (resp. $N_G[x]$) is called the open (resp. closed) neighbourhood of $x$ in $G$. For each subset $S$ of $V$, the induced subgraph of $G$ on $V\setminus S$ is denoted by $G\setminus S$. $G\setminus \{x\}$ is briefly denoted by $G\setminus x$. The complete graph, the path graph and the cycle graph with $n$ vertices are denoted by $K_n$, $P_n$ and $C_n$ respectively. A graph with no vertices and edges is called the empty graph.  A graph which has no $C_n$ as a subgraph is called $C_n$-free. A subset $S$ of $V(G)$ is called an independent set of vertices in $G$ if the induced subgraph of $G$ on $S$ is totally disconnected. The maximum size of independent sets of vertices in $G$ is denoted by $\alpha(G)$ and if all maximal independent sets of vertices in $G$ have the same size, $G$ is called unmixed. The independence complex $\Delta_G$ of $G$ is the simplicial complex with vertex set $V$ whose faces are the independent sets of vertices of $G$. A graph is called vertex decomposable (resp. shellable, Cohen-Macaulay, unmixed) if its independence complex is vertex decomposable (resp. shellable, Cohen-Macaulay, pure). These properties of graphs have been studied in many papers (cf. \cite{D+E, F+V, G+W, HH, M+K+Y, Vantuyl, W}).  In view of \cite[Lemma 4]{W}, a graph $G$ is vertex decomposable if either it is totally disconnected or there is a vertex $x$ in $G$ such that 
\begin{itemize}
\item[1.] $G\setminus x$ and $G\setminus N_G[x]$ are both vertex decomposable, and
\item[2.]  no independent set of vertices in $G\setminus N_G[x]$ is a maximal independent set of vertices in $G\setminus x$.
\end{itemize}
The vertex $x$ satisfying the condition 2 above is called a shedding vertex of $G$. Also a simplicial complex $\Delta$ is shellable if the facets of $\Delta$ can be ordered as $F_1, \dots , F_m$ in such a way that for all $1\leq i<j\leq m$, there exist $v\in F_j\setminus F_i$ and $\ell\in \{1, \dots , j-1\}$ with $F_j\setminus F_\ell=\{v\}$. Such an order $F_1, \dots , F_m$ is called a shelling of $\Delta$. Moreover for a fixed shelling and two facets $F$ and $F'$ of $\Delta$ we write $F<F'$ if $F$ is lied before $F'$ in the shelling. For our convenience we use the same name face (facet) of $G$ for a (maximal) independent set of vertices of $G$ which are in fact the faces (facets) of $\Delta_G$. 

Because of the following known implications, finding new vertex decomposable and shellable graphs leads us to find new (sequentially) Cohen-Macaulay graphs:
$$\textrm{Vertex decomposable} \Rightarrow \textrm{Shellable} \Rightarrow \textrm{Sequentially Cohen-Macaulay},$$
$$\textrm{Pure shellable} \Rightarrow \textrm{Cohen-Macaulay}.$$ 
Among a lot of works concerning to these combinatorial and algebraic properties of graphs, some of them are devoted to construct new graphs with these properties from previous ones (cf. \cite{D+E, H+T, M+K+Y, Mousi, Vantuyl}).  In this regard it is natural to study vertex decomposability (shellability, Cohen-Macaulayness) under graph operations. In this paper, we concentrate on some well-known graph operations such as disjoint union, join, rooted product, corona, cartesian product (product) and lexicographic product (composition). We remark that some of them (disjoint union and rooted product) are studied through some other works (\cite{Mousi, VV, W}).  Moreover \cite[Theorem 3.12]{SF} and \cite[Theorem 1.1]{Vantuyl} have been generalized in Theorems \ref{shell}, \ref{shell2} and \ref{2.15}. Also among the other things, by using these graph operations, some new classes of Cohen-Macaulay graphs from previous ones are presented in Proposition \ref{pro2.8} and Corollaries \ref{cor2.7} and \ref{cor2.16}.
%For more details and terminologies in graph theory we refer the reader to \cite{Berge}.

\section{Upon some graph operations}
 Throughout this paper, $G$ and $H$ are two graphs with disjoint sets of vertices and $\mathcal{H}=\{H_x \ | \ x\in V(G)\}$ is a family of graphs indexed by vertices of $G$. Furthermore, all graphs are non-empty, finite and simple unless otherwise is stated.  There are well-known results about vertex decomposability and shellability of cycles and chordal graphs which may be used in our examples in the sequel:

\begin{theorem}\label{2.1}
\begin{itemize}
\item[1.] \cite[Proposition 4.1]{F+V} $C_n$ is vertex decomposable (shellable) if and only if $n=3$ or $n=5$;
\item[2.] \cite[Theorem 1.2]{VV} and \cite[Corollary 7]{W} Any chordal graph is vertex decomposable (shellable).
\end{itemize}
\end{theorem} 
The main body of our work is organized into the following subsections:

\subsection{Disjoint union ($G\cup H$).} 
Disjoint union $G\cup H$ of $G$ and $H$ is a graph with $V(G\cup H) =V(G)\cup V(H)$ and $E(G\cup H)=E(G)\cup E(H)$.

The following lemma plays an essential role in our results.
\begin{lemma}\label{lemma2.4}
\cite[Lemma 2.4]{VV} and \cite[Lemma 20]{W}
$G$ and $H$ are vertex decomposable (shellable) if and only if $G\cup H$ is vertex decomposable (shellable).
\end{lemma}

\subsection{Join ($G+H$).} 

The join $G+H$ of $G$ and $H$ defined by Zykov in \cite{Z} is disjoint union of $G$ and $H$ with additional edges joining each vertex of $G$ to each vertex of $H$.

 Since  $\Delta_{G+H}=\Delta_G\cup \Delta_H$, one can easily gain the following proposition in which a necessary and sufficient condition for vertex decomposability (shellability) of $G+H$ is given.

\begin{proposition}\label{2.2}
%\begin{itemize}
%\item[1.] For each $m\in \mathbb{N}$ and $x\in V(G+K_m)$,  $x$ is a shedding vertex of $G+K_m$ if and only if either $x\in V(K_m)$ or $x$ is a shedding vertex of $G$.
%\item[2.]
 $G+H$ is vertex decomposable (shellable) if and only if $G$ and $H$ are vertex decomposable (shellable) and at least one of them is complete.
%\end{itemize}
\end{proposition}
\begin{corollary}
Suppose that $G+H$ is Cohen-Macaulay. Then 
\begin{itemize}
\item[1.] $\alpha(G)=\alpha(H)$;
\item[2.] $G+H$ is vertex decomposable if and only if $G+H$ is complete.
\end{itemize}
\end{corollary}
\begin{proof}
1. follows from  \cite[Corollary 5.3.12]{V2} and \cite[Corollary 6]{T+V}.

2. follows from Theorem \ref{2.1}(2), Proposition \ref{2.2} and Part 1.
\end{proof}
\subsection{Rooted product ($G(\mathcal{H})$).}

The rooted product $G(\mathcal{H})$ of a graph $G$ by a family of rooted graphs $\mathcal{H}=\{H_x \ | \ x\in V(G)\}$ is defined in 1978, \cite{G+M}, as the union of $G$ and $H_x$s such that for each vertex $x$ of $G$ one should identify $x$ with the root of $H_x$. If all graphs of the family $\mathcal{H}$ are isomorphic to one graph, say $H$, then it is called the rooted graph of $G$ by $H$ and is denoted by $G(H)$.  
%In fact one may see that the rooted graph of $G$ and $H$ is a subgraph of the cartesian product of $G$ and $H$.

The authors in \cite{Mousi} have studied the rooted product $G(\mathcal{H})$ as a new construction of graphs and have gained the following result:

\begin{theorem}\cite[Proposition 2.3 and Remark 2.4]{Mousi}\label{rooted}
Suppose that $G$ is a simple graph and $\mathcal{H}=\{H_x \ | \ x\in V(G)\}$ is a family of rooted graphs such that $x$ is the root vertex of $H_x$. Then
\begin{itemize}
\item[1.] If $H_x$s are vertex decomposable and $x$ is a shedding vertex of $H_x$ for all $x\in V(G)$, then $G(\mathcal{H})$ is also vertex decomposable.
\item[2.]  If $G(\mathcal{H})$ is vertex decomposable (shellable), then $H_x$s are vertex decomposable (shellable).
\end{itemize}
\end{theorem}

Also, in \cite[Theorem 2]{Kengo}, shellability of $G(H)$ is investigated.

Note that $C_4(K_2)$ is a vertex decomposable and so shellable graph, while $C_4$ is neither vertex decomposable nor shellable. So vertex decomposability (shellability) of $G(H)$ doesn't necessarily imply the vertex decomposability (shellability) of $G$. Also \cite[Example 2.5]{Mousi} shows that the converse of Theorem \ref{rooted}(2) is not generally true for shellability or vertex decomposability when $x$ is not shedding for some $x\in V(G)$.

\subsection{Corona ($Go\mathcal{H}$).}
The corona $Go\mathcal{H}$ of $G$ and $\mathcal{H}$ is the disjoint union of $G$ and $H_x$s with additional edges joining each vertex $x$ of $G$ to all
vertices of $H_x$. If all graphs of the family $\mathcal{H}$ are isomorphic to one graph, say $H$, then we shall write $G o H$ instead of $G o \mathcal{H}$. %(For our definition of corona, $H_x$ may be the empty graph for each $x\in V(G)$.)

One may find out the corona of $G$ and $\mathcal{H}$ is in fact the rooted product $G(\mathcal{H}')$ where $\mathcal{H}'=\{K_1+H_x \ | \ x\in V(G)\}$ and the only vertex of $K_1$ is the root of $K_1+H_x$ for each $x\in V(G)$. Also both of them are generalizations of adding whiskers or complete graphs to a graph which is studied in \cite{F+H, HH, V}. 

In the following result, by means of some shellings of $H_x$s we shall present a shelling for $Go\mathcal{H}$. To this aim we need to remark that each facet of $Go\mathcal{H}$ is in form of $F\cup (\bigcup_{\ell =1}^kF_{i_\ell})$, such that $F$ is a face of $G$ and $F_{i_\ell}$ is a facet of $H_{x_{i_\ell}}$ when $V(G)\setminus F=\{x_{i_1}, \dots , x_{i_k}\}$. We call this facet as a facet associated to the face $F$ of $G$.

\begin{theorem}\label{theorem2.4}
%\begin{itemize}
%\item[1.] Suppose that $x\in V(G)$. Then $x$ and each shedding vertex of $H_x$ are shedding vertices of $Go\mathcal{H}$.
%\item[2.]
  $Go\mathcal{H}$ is vertex decomposable (shellable) if and only if $H_x$ is vertex decomposable (shellable) for all $x\in V(G)$. 
%\end{itemize}
\end{theorem}
\begin{proof}
%1. Assume that $N_G(x)=\{x_1, \dots , x_k\}$. It can be easily seen that
%\begin{align}\label{eq2}
%&(Go\mathcal{H})\setminus x=((G\setminus x)o\mathcal{H}') \cup H_x,\\
%\nonumber &(Go\mathcal{H})\setminus N_{Go\mathcal{H}}[x]=((G\setminus N_G[x])o\mathcal{H}'') \cup \bigcup_{i=1}^k H_{x_i},
%\end{align}
%where  $\mathcal{H}'=\mathcal{H}\setminus \{H_x\}$ and  $\mathcal{H}''=\mathcal{H} \setminus \{H_y \ | \ y\in N_G[x]\}$. Now, if $S$ is an independent set of vertices in $(Go\mathcal{H})\setminus N_{Go\mathcal{H}}[x]$. It is clear that for each $y\in H_x$, $S\cup \{y\}$ is also independent in $(Go\mathcal{H})\setminus  x$. So $x$ is a shedding vertex of $Go\mathcal{H}$.

%Suppose that $y$ is a shedding vertex of $H_x$. Then 
%\begin{align}\label{eqq2}
%&(Go\mathcal{H})\setminus y= Go\mathcal{H}',\\
%\nonumber &(Go\mathcal{H})\setminus N_{Go\mathcal{H}}[y] = ((G\setminus x)o\{H_{x'} \ | \ x'\in V(G), x'\neq x\}) \cup (H_x\setminus N_{H_x}[y]),
%\end{align}
%where $\mathcal{H}'=\{H_x\setminus y\}\cup \{H_{x'} \ | \ x'\in V(G), x'\neq x\}$. Assume that $S$ is a maximal independent set of vertices in $(Go\mathcal{H})\setminus N_{Go\mathcal{H}}[y]$. Then $S=S_1\cup S_2$ for some maximal independent sets of vertices $S_1$ in $(G\setminus x)o\{H_{x'} \ | \ x'\in V(G), x'\neq x\}$ and   $S_2$ in $H_x\setminus N_{H_x}[y]$. Hence there exists a vertex $z\in N_{H_x}(y)$ such that $S_2\cup \{z\}$ is independent in $H_x\setminus \{y\}$. Thus clearly $S\cup \{z\}$ is also independent. So $y$ is a shedding vertex of $Go\mathcal{H}$ as desired.

Vertex decomposability and the \textit{only if} part of shellability can be immediately gained from Proposition \ref{2.2} and Theorem \ref{rooted}. So it remains to prove the \textit{if part} of shellability:

 If $G$ is totally disconnected, then the result holds by Lemma \ref{lemma2.4} and Proposition \ref{2.2}. Else each facet of $Go\mathcal{H}$  contains at least one facet of $H_x$ for some vertex $x$ of $G$.
Suppose that $V(G)=\{x_1, \dots , x_n\}$ and for each $1\leq i \leq n$,  $F_{i,1}, \dots , F_{i,m_i}$ is a shelling for $H_{x_i}$. Suppose that $K=F\cup (\bigcup_{\ell =1}^kF_{i_\ell})$ and $K'=F'\cup (\bigcup_{\ell =1}^sF_{j_\ell})$ are two facets of $Go\mathcal{H}$ such that $V(G)\setminus F=\{x_{i_1}, \dots , x_{i_k}\}$ and $V(G)\setminus F'=\{x_{j_1}, \dots , x_{j_s}\}$ where $1\leq i_1<\dots <i_k\leq n$ and $1\leq j_1< \dots < j_s \leq n$. Now if $i_1<j_1$ set $K<K'$ and if $i_1=j_1$ and $F_{i_1}<F_{j_1}$ in $H_{x_{i_1}}$, then set $K<K'$. Else we have $i_1=j_1$ and $F_{i_1}=F_{j_1}$. Then if $i_2< j_2$ we set $K<K'$ and if $i_2=j_2$ and $F_{i_2}<F_{j_2}$ in $H_{x_{i_2}}$, then set $K<K'$. Else we have $i_1=j_1, i_2=j_2, F_{i_1}=F_{j_1}$ and $F_{i_2}=F_{j_2}$. So by going on in this way, we can order $K$ and $K'$. Note that if $F$ is properly contained in $F'$  and $F_{i_\ell}=F_{j_\ell}$ for all $1\leq \ell \leq s$, then set $K<K'$.

%We order the facets of $Go\mathcal{H}$ by the lexicographic order as  follows:
%\begin{itemize}
%\item For $1\leq i<j\leq n$ and $1\leq \ell \leq m_i$, each facet of $Go\mathcal{H}$ containing $F_{i,\ell}$ is lied before the facets containing $F_{j, \ell '}$ for some $1\leq \ell' \leq m_j$ and not containing $F_{i, \ell''}$ for all $1\leq \ell'' \leq m_i$. 
%\item For each $1\leq i \leq n$ and $1\leq \ell<\ell ' \leq m_i$, the facets containing $F_{i, \ell}$ is lied before the facets containing $F_{i, \ell'}$.
%\end{itemize}
Now we claim that this order forms a shelling of $Go\mathcal{H}$. Assume that $K$ and $K'$ are two facets of $Go\mathcal{H}$ with $K<K'$ which are respectively associated to the faces $F$ and $F'$ of $G$. In each of the following cases we shall obtain the desired things:

Case I. $F=F'$. Then $K=F\cup (\bigcup_{\ell =1}^kF_{i_\ell, r_\ell})$ and $K'=F\cup (\bigcup_{\ell =1}^kF_{i_\ell, r'_\ell})$, where $V(G)\setminus F=\{x_{i_1}, \dots , x_{i_k}\}$, $1\leq r_\ell, r'_\ell \leq m_{i_\ell}$. By our order there exists $1\leq j \leq k$ such that $r_1=r'_1, \dots , r_{j-1}=r'_{j-1}$ and $r_j<r'_j$. Therefore $F_{i_j, r_j}<F_{i_j, r'_j}$ in the shelling of $H_{x_{i_j}}$. Hence there exist a vertex $v\in F_{i_j, r'_j}\setminus F_{i_j, r_j}$ and a facet $F_{i_j, r''_j}$ such that $r''_j<r'_j$,  $F_{i_j, r'_j}\setminus F_{i_j, r''_j}=\{v\}$. Now, set $K''=F\cup F_{i_j,r''_j} \cup (\bigcup_{\ell =1, \ell\neq j}^kF_{i_\ell, r'_\ell})$. Then it is clear that $K''<K', v\in K'\setminus K$ and $K'\setminus K''=\{v\}$.

Case II. $F\neq F'$. Then $K=F\cup (\bigcup_{\ell =1}^kF_{i_\ell, r_\ell})$ and $K'=F'\cup (\bigcup_{\ell =1}^{k'}F_{i'_\ell, r'_\ell})$, where $V(G)\setminus F=\{x_{i_1}, \dots , x_{i_k}\}$, $V(G)\setminus F'=\{x_{i'_1}, \dots , x_{i'_{k'}}\}$, $1\leq r_\ell \leq m_{i_\ell}$ and $1\leq r'_\ell \leq m_{i'_\ell}$. Since $F\neq F'$ and $K<K'$, there exists a vertex $x_{i_{\ell_0}}\in F'\setminus F$ and for each $1\leq \ell <i_{\ell_0}$, $x_\ell\in F'$ implies that $x_\ell\in F$. Set $F''=F'\setminus \{x_{i_{\ell_0}}\}$ and $K''=F''\cup F_{i_{\ell_0}, r_{\ell_0}}\cup (\bigcup_{\ell =1}^{k'}F_{i'_\ell, r'_\ell})$. Then it is clear that $K''<K', x_{i_{\ell_0}}\in K'\setminus K$ and $K'\setminus K''=\{x_{i_{\ell_0}}\}$.
\end{proof}
\begin{corollary}\label{cor2.7}
$Go\mathcal{H}$ is a Cohen-Macaulay graph if and only if $H_x$ is complete for all $x\in V(G)$.
\end{corollary}
\begin{proof}
$(\Rightarrow)$ \cite[Corollary 5.3.12]{V2} and \cite[Theorem 1]{T+V} yield the result.

$(\Leftarrow)$ \cite[Theorem 1]{T+V} and Theorems \ref{2.1}  and \ref{theorem2.4} show that $Go\mathcal{H}$ is unmixed and shellable. Thus the result holds by Theorem 5.3.18  in \cite{V2}.
\end{proof}

Note that if $Go\mathcal{H}$ is vertex decomposable (shellable, Cohen-Macaulay), it doesn't need $G$ is vertex decomposable (shellable, Cohen-Macaulay). For instance, suppose that $G=C_4$ and $H=K_1$. Then by Theorem \ref{theorem2.4} and Corollary \ref{cor2.7} $GoH$ is vertex decomposable, shellable and Cohen-Macaulay but $G$ is none of them.

\subsection{Cartesian product ($G \square H$).}

The cartesian product $G \square H$ of $G$ and $H$ is a graph with $V(G\square H)= V(G) \times V(H)$ and
two vertices $(u,u' )$ and $(v,v' )$ are adjacent in $G \square H$ if and only if either
$u = v$ and $u'$ is adjacent to $v'$ in $H$, or
$u' = v'$ and $u$ is adjacent to $v$ in $G$.

Since cartesian product is a commutative operation and $G \square K_1=G$ for any graph $G$, we may assume that $|V(G)|, |V(H)|\geq 2$. Also, we know that $K_2 \square K_2=C_4$ which is not vertex decomposable (shellable). So even if both of $G$ and $H$ are vertex decomposable (shellable), $G\square H$ does not need to be vertex decomposable (shellable). Of course it is known that the cartesian product of two graphs are often fails to be vertex decomposable (shellable), but the following proposition shows that it is not impossible. So, next we are going to find a necessary condition for vertex decomposability of $G\square H$ which conclude a result about the girth of $G$ and $H$. 

\begin{proposition}\label{pro2.8}
For each integer $m\geq 3$, $K_m\square K_2$ is unmixed and vertex decomposable and so Cohen-Macaulay.
\end{proposition}
\begin{proof}
In view of Proposition 15 in \cite{T+V}, we only need to prove the vertex decomposability of $G=K_m\square K_2$ where $m\geq 3$ and $V(G)=\{(i,j) \ | \ 1\leq i \leq m, 1\leq j \leq 2\}$. We proceed by induction on $m$. For $m=3$ it can be easily checked that $G$ is vertex decomposable. Now suppose that $m>3$ and for each integer $3\leq m'<m$, $K_{m'}\square K_2$ is vertex decomposable. Note that $G\setminus N_G[(1,1)]$ is the complete graph with vertex set $\{(i,2) \  | \ 2\leq i \leq m\}$ and so is vertex decomposable. Also $G\setminus (1,1)$ is $K_{m-1}\square K_2$ with additional vertex $(1,2)$ such that it is joined to each vertex $(i,2)$ for $2\leq i \leq m$. Since $m>3$, for each maximal independent set $\{(i,2)\}$ in  $G\setminus N_G[(1,1)]$, there is a vertex $(j,1)$ such that $\{(i,2), (j,1)\}$ is independent in $G\setminus (1,1)$. Now we need to show that $G'=G\setminus (1,1)$ is vertex decomposable. It is easy to see that the set of vertices of $G'$ with minimal closed neighbourhood induce a disjoint union of a complete graph whose set of vertices is $\{(i,1)|2\leq i\leq m\}$ and a vertex $(1, 2)$. So $G'$ is vertex decomposable by Theorem 4.2 in \cite{G+W}.

%It is clear that $G'\setminus (1,2)=K_{m-1}\square K_2$ which is vertex decomposable by inductive hypothesis and $G'\setminus N_{G'}[(1,2)]$ is the complete graph with $m-1$ vertices which is vertex decomposable. Also since $m>3$, one can see that $(1,2)$ is a shedding vertex of $G'$. These show that $(1,1)$ is a shedding vertex for $G$ which completes the proof.
\end{proof}
 
For our next result we need the following lemma.
\begin{lemma}\label{2.4}
If $(x,y)$ is a shedding vertex of $G\square H$, then $x$ is a shedding vertex of $G$ or $y$ is a shedding vertex of $H$.
\end{lemma}
\begin{proof}
Suppose in contrary that $(x,y)$ is a shedding vertex of $K=G\square H$ but neither $x$ is a shedding vertex of $G$ nor $y$ is a shedding vertex of $H$.
Assume that $N_G(x)=\{x_1, \dots , x_m\}$ and $N_H(y)=\{y_1, \dots , y_n\}$. Then there exist independent sets of vertices $S_1$ in $G\setminus N_G[x]$ and $S_2$ in $H\setminus N_H[y]$ such that for each $1\leq i \leq m$ and $1\leq j \leq n$, $S_1\cup \{x_i\}$ and $S_2\cup \{y_j\}$ are not independent. Hence 
$S=\{(z,y) \ | \ z\in S_1\}\cup \{(x,z) \ | \ z\in S_2\}$ is an independent set of vertices in $K \setminus N_K[(x,y)]$. Since $S_1\cup \{x_i\}$ is not independent, for each $1\leq i \leq m$,  $(z,y)$ is adjacent to $(x_i,y)$ for some $z\in S_1$. Similarly since $S_2\cup \{y_j\}$ is not independent,  for each $1\leq j \leq n$, $(x,z')$ is adjacent to $(x,y_j)$ for some $z'\in S_2$. Hence since 
$$N_K((x,y))=\{(x, y_1), \dots , (x, y_n), (x_1, y), \dots , (x_m, y)\},$$
 if $S'$ is a maximal independent set of vertices in $K\setminus N_K[(x,y)]$ containing $S$, then it is also a maximal independent set of vertices in $K\setminus (x,y)$ which is a contradiction.
\end{proof}

\begin{proposition}\label{2.5}
Suppose that $G$ and $H$ are simple graphs with no isolated vertices in which the neighbourhood of each shedding vertex is independent. Then $G\square H$ doesn't have any shedding vertex and so it is not vertex decomposable.
\end{proposition}
\begin{proof}
Set $K=G\square H$. Suppose in contrary that $K$ has a shedding vertex, say $(x,y)$. 
Assume that $N_G(x)=\{x_1, \dots , x_m\}$ and $N_H(y)=\{y_1, \dots , y_n\}$. Then 
$$N_K((x,y))=\{(x, y_1), \dots , (x, y_n), (x_1, y), \dots , (x_m, y)\}.$$
In the light of Lemma \ref{2.4} we may consider the following two cases:

Case I. $x$ and $y$ are shedding vertices of $G$ and $H$ respectively. 

By our assumption
$S=\{(x_i, y_j) \ | \ 1\leq i \leq m , 1\leq j \leq n\}$
is an independent set of vertices in $K\setminus N_K[(x,y)]$. But for each $1\leq i \leq m$ and $1\leq j \leq n$, $(x,y_j)$ and $(x_i,y)$ are adjacent to $(x_i,y_j)$. 

Case II. $x$ is a shedding vertex of $G$ and $y$ is not a shedding vertex of $H$. 

Then $N_G(x)$ is an independent set of vertices in $G$ and there exists an independent set of vertices $S_2$ in $H\setminus N_H[y]$ such that for each $1\leq j \leq n$, $S_2\cup \{y_j\}$ is not independent. Hence 
$S=\{(x_i,y_1) \ | \ 1\leq i \leq m\}\cup \{(x,z) \ | \ z\in S_2\}$ is an independent set of vertices in $K\setminus N_K[(x,y)]$. Since $S_2\cup \{y_j\}$ is not independent,  for each $1\leq j \leq n$, $(x,y_j)$ is adjacent to $(x,z)$ for some $z\in S_2$. Also, it is clear that $(x_i,y)$ is adjacent to $(x_i,y_1)$ for each $1\leq i \leq m$. 
 
 Hence in each status if $S'$ is a maximal independent set of vertices in $K\setminus N_K[(x,y)]$ containing $S$, then it is also a maximal independent set of vertices in $K\setminus (x,y)$ which is a contradiction.
\end{proof}

It is evident that a graph $G$ is $C_3$-free if and only if the neighbourhood of any vertex in $G$ is independent. The minimum length of cycles in $G$ is called its girth and is denoted by $\mathrm{girth}G$. The following corollary, which gives some information about the girth of two graphs by means of vertex decomposability of their cartesian product, is an immediate consequence of Proposition \ref{2.5}.

\begin{corollary}\label{cor2}
If $G$ and $H$ are two $C_3$-free graphs with no isolated vertices,  then $G\square H$ is not vertex decomposable. Hence if $G$ and $H$ have no isolated vertex and $G\square H$ is vertex decomposable, then
$$\min\{\mathrm{girth}G, \mathrm{girth}H\}=3.$$
\end{corollary}

Although in Corollary \ref{cor2} we prepare a necessary condition for vertex decomposability of $G \square H$, but this is not a sufficient condition. For instance suppose that $K=C_3 \square C_3$. Then for each vertex $(x,y)$ in $K$, $K\setminus N_K[(x,y)]$ is a 4-cycle. Hence it is not vertex decomposable, so is not $K$. Note that in this case none of the vertices of $K$ is a shedding vertex.

The following corollary, which investigates the cartesian product of cycles, immediately follows from Corollary \ref{cor2}, \cite[Corollary 5.3.12]{V2} and \cite[Corollary 12]{T+V}.
\begin{corollary}
If $C_n\square C_m$ is vertex decomposable or Cohen-Macaulay, then $n=3$ or $m=3$.
\end{corollary}

\subsection{Lexicographic product ($G[\mathcal{H}]$).}
The lexicographic product $G[\mathcal{H}]$ of $G$ and $\mathcal{H}$, which was first studied by Hausdorff \cite{H}, is a graph with 
$$V(G[\mathcal{H}])=\bigcup_{x\in V(G)} \{(x, y) \ | \  y\in V(H_x)\}.$$
Two vertices $(x,y)$ and $(x',y')$ are adjacent if either $x=x'$ and $y$ is adjacent to $y'$ in $H_x$ or $x$ is adjacent to $x'$ in $G$. If all the graphs in $\mathcal{H}$ are isomorphic to one graph, say $H$, then $G[\mathcal{H}]$ is denoted by $G[H]$.

In \cite[Theorem 1.1]{Vantuyl} (see also \cite[Lemma 4.51]{Hoshino}) vertex decomposability and shellability of $G[H]$ is studied in pure case. More precisely for its shellability, its well-coveredness  is used and then vertex decomposability is studied via shellability. Here we establish our results without the strong condition of well-coveredness. Also we attend $G[\mathcal{H}]$ in stead of $G[H]$ which covers a larger class of graphs. We start with shellability. It is clear that all facets of $G[\mathcal{H}]$ are in form of $(\{x_{i_1}\}\times F_{i_1})\cup \dots \cup (\{x_{i_k}\}\times F_{i_k})$ for some facet  $F=\{x_{i_1}, \dots , x_{i_k}\}$ of $G$ such that for each $1\leq r \leq k$, $F_{i_r}$ is a facet of $H_{x_{i_r}}$. We call it as a facet associated to the facet $F$. So the following result is a generalization of shellable part of \cite[Theorem 1.1]{Vantuyl}, \cite[Lemma 4.51]{Hoshino} and \cite[Example 28]{Kengo}.

\begin{theorem}\label{shell}
\begin{itemize}
\item[1.] If $G$ is totally disconnected, then $G[\mathcal{H}]$ is shellable if and only if $H_x$ is shellable for all $x\in V(G)$.
\item[2.] If $G[\mathcal{H}]$ is shellable, then $H_x$ is shellable for all $x\in V(G)$.
\item[3.] If $G[\mathcal{H}]$ is shellable, then for each edge $\{x,y\}$ of $G$,  $H_x$ is complete or $H_y$ is complete.
\end{itemize} 
\end{theorem}
\begin{proof}
1. In this case it is obvious that $G[\mathcal{H}]$ is isomorphic to disjoint union of $H_x$s. Hence the result holds by Lemma \ref{lemma2.4}. 

2. In view of Lemma \ref{lemma2.4} we may assume that $G$ is connected. Fix $x\in V(G)$ and a facet  $F=\{x_{i_1}, \dots , x_{i_k}\}$ of $G$ containing $x$. Without loss of generality assume $x_{i_1}=x$. Suppose that $F_1, \dots, F_m$ are all facets of $H_x$ such that for a fixed shelling of $G[\mathcal{H}]$ and fixed facets $F_{i_2},\dots , F_{i_k}$ of $H_{x_{i_2}}, \dots , H_{x_{i_k}}$ respectively, we have 
\begin{align*}
F_1'&=(\{x\}\times F_1) \cup (\{x_{i_2}\}\times F_{i_2})\cup \dots \cup (\{x_{i_k}\}\times F_{i_k})<\\
F'_2&=(\{x\}\times F_2) \cup (\{x_{i_2}\}\times F_{i_2})\cup \dots \cup (\{x_{i_k}\}\times F_{i_k})<\\
&\vdots \\
F_m'&=(\{x\}\times F_m) \cup (\{x_{i_2}\}\times F_{i_2})\cup \dots \cup (\{x_{i_k}\}\times F_{i_k}).
\end{align*}
We claim that $F_1, \dots , F_m$ is a shelling of $H_x$. Suppose $F_i<F_j$. Then $F_i'<F_j'$. Hence there is a vertex $(x', y')\in F'_j\setminus F'_i$ and a facet $F'<F'_j$ such that $F'_j\setminus F'=\{(x',y')\}$. It is clear that $x'=x$ and $y'\in F_j\setminus F_i$. If $F'$ is associated to a facet of $G$ that does not contain $x$, then we have  $\{x\}\times F_j\subseteq F'_j\setminus F'=\{(x,y')\}$. This implies that $F_j= \{y'\}$ and so $F_j\setminus F_i=\{y'\}$ as desired. Else, $F'$ is a facet of $G[\mathcal{H}]$ associated to a facet of $G$ containing $x$.  Then we should have $F'=(\{x\}\times F_\ell) \cup S$  for some $\ell$ and $S$. Since $F'_j\setminus F'=\{(x,y')\}$, we have
$$\{(x,y')\}=(\{x\}\times (F_j\setminus F_\ell))\cup \left( (\{x_{i_2}\}\times F_{i_2})\cup \dots \cup (\{x_{i_k}\}\times F_{i_k})\setminus S\right).$$
Therefore $F_j\setminus F_\ell=\{y'\}$ and $S= (\{x_{i_2}\}\times F_{i_2})\cup \dots \cup (\{x_{i_k}\}\times F_{i_k})$ since $F$ and $F_{i_r}$s are facets. Hence $F'=F_\ell'$ for some $1\leq \ell < j$ as desired.  

3. In the light of Lemma \ref{lemma2.4} we may assume that $G$ is connected. Assume that $G[\mathcal{H}]$ is shellable and in contrary there are adjacent vertices $x_1$ and $x_2$  of $G$ such that $H_{x_1}$ and $H_{x_2}$ are not complete. Since $H_{x_1}$ (resp. $H_{x_2}$) is not complete, it has a facet with more than one vertex, say $F_1$ (resp. $F_2$). Now assume that in a fixed shelling of $G[\mathcal{H}]$, $K_1$ (resp. $K_2$) is the first facet of $G[\mathcal{H}]$ containing $\{x_1\}\times F_1$ (resp. $\{x_2\}\times F_2$). Since $x_1$ is adjacent to $x_2$, $K_1$ and $K_2$ are distinct. If $K_1<K_2$, then there should exist a facet $K$ of $G[\mathcal{H}]$ with $K<K_2$ such that $|K_2\setminus K|=1$. Now two cases occure:

$\bullet$ If $K$ includes $\{x_2\}\times F_2'$ for some facet $F_2'$ of $H_{x_2}$, then since $|F_2|>1$, $F_2'$ is a facet and 
$$|F_2\setminus F_2'|=|\{x_2\}\times (F_2\setminus F_2')|\leq |K_2\setminus K|=1,$$
$F_2'$ should have more than one vertex which contradicts with our assumption on $K_2$. 

$\bullet$ If $K$ is associated to a facet of $G$ that does not contain $x_2$, then 
$$1<|F_2|=|\{x_2\}\times F_2|\leq |K_2\setminus K|=1,$$
which is again a contradiction.

Hence $K_1<K_2$ is impossible. By similar argument one can see that $K_2<K_1$ is also impossible. This completes the proof.
\end{proof}

The following corollary can be immediately follows from Theorem \ref{shell}. Recall that a subset $S$ of $V(G)$ is a vertex cover of $G$ if each edge of $G$  has at least one vertex in $S$. Also the vertex covering number of $G$, $\alpha (G)$, is the minimum cardinality of vertex covers of $G$.

\begin{corollary}
\begin{itemize}
\item[1.] If $|\{x\in V(G) \ | \ H_x \text{ is complete}\}|<\alpha (G)$, then $G[\mathcal{H}]$ is not shellable.

\item[2.] If $G[\mathcal{H}]$ is shellable, then $\{x\in V(G) \ | \ H_x \text{ is complete}\}$ is a vertex cover of $G$.
\end{itemize}
\end{corollary}

In \cite[Theorem 3.12]{SF} it is shown that if $G$ is a shellable graph and $H_x$s are complete, then $G[\mathcal{H}]$ is also shellable. Note that its converse is not generally true, because for example it can be checked that $C_4[\{K_2, K_1, K_1, K_1\}]$ and $K_2[P_3,K_2]$ are shellable, while $C_4$ is not shellable and $P_3$ is not complete. 

In the following result we provide circumstances under which $G[\mathcal{H}]$ is shellable.  Examples \ref{example} illustrate that this result is  a generalization of Theorem 3.12 in \cite{SF}.

\begin{theorem}\label{shell2}
Suppose that $H_x$ is shellable for every $x\in V(G)$ and $G$ is a shellable graph with a shelling in which for all facets $F<F'$ there exist a vertex $u\in F'\setminus F$ and a facet $F''<F'$ such that $F'\setminus F''=\{u\}$ and $H_u$ is complete. Then $G[\mathcal{H}]$ is also shellable.
\end{theorem}
\begin{proof}
Suppose $V(G)=\{x_1, \dots, x_n\}$ and for every $1\leq i \leq n$ fix a shelling for $H_{x_i}$. We order the facets of $G[\mathcal{H}]$ as follows:
\begin{itemize}
\item For all facets $F<F'$ of $G$, each facet of $G[\mathcal{H}]$ associated to $F$ is lied before the facets associated to $F'$;
\item For every two facets $K=(\{x_{i_1}\}\times F_{i_1})\cup \dots \cup (\{x_{i_k}\}\times F_{i_k})$ and $K'=(\{x_{i_1}\}\times F'_{i_1})\cup \dots \cup (\{x_{i_k}\}\times F'_{i_k})$ associated to the facet $F=\{x_{i_1}, \dots , x_{i_k}\}$ of $G$, $K$ is lied before $K'$ if there is an integer  $1\leq \ell \leq k$ such that $F_{i_1}=F'_{i_1}, \dots , F_{i_{\ell-1}}=F'_{i_{\ell-1}}$ and $F_{i_\ell}<F'_{i_\ell}$. 
\end{itemize}
Now we claim that this order forms a shelling for $G[\mathcal{H}]$. Assume that $K$ and $K'$ are two facets of $G[\mathcal{H}]$ with $K<K'$ which are respectively associated to the facets $F$ and $F'$ of $G$. The following two cases occur:

Case I. $F=F'=\{x_{i_1}, \dots , x_{i_k}\}$.  If  $K=(\{x_{i_1}\}\times F_{i_1})\cup \dots \cup (\{x_{i_k}\}\times F_{i_k})$ and $K'=(\{x_{i_1}\}\times F'_{i_1})\cup \dots \cup (\{x_{i_k}\}\times F'_{i_k})$, then there is an integer $1\leq \ell \leq k$ such that $F_{i_1}=F_{i_1}', \dots , F_{i_{\ell-1}}=F'_{i_{\ell-1}}$ and $F_{i_\ell}<F'_{i_\ell}$. Since $H_{x_{i_\ell}}$ is shellable, there exist $y\in F'_{i_\ell}\setminus F_{i_\ell}$ and $F''_{i_\ell}<F'_{i_\ell}$ such that $F'_{i_\ell}\setminus F''_{i_\ell}=\{y\}$. Set 
$$K''=(\{x_{i_1}\}\times F'_{i_1})\cup \dots \cup (\{x_{i_{\ell -1}}\}\times F'_{i_{\ell-1}})\cup (\{x_{i_\ell}\}\times F''_{i_\ell})\cup (\{x_{i_{\ell +1}}\}\times F'_{i_{\ell+1}})\cup \dots \cup (\{x_{i_k}\}\times F'_{i_k}).$$
It is seen $K''<K'$, $(x_{i_\ell},y)\in K'\setminus K$ and $K'\setminus K''=\{(x_{i_\ell}, y)\}$ as required.

Case II. $F\neq F'$. Then $F<F'$. So by our assumption there exist $u\in F'\setminus F$ and $F''<F'$ such that $F'\setminus F''=\{u\}$ and $H_u$ is complete. Thus all facets of $H_u$ are singleton. Suppose that $F'=\{u, x_{i_2}, \dots , x_{i_k}\}$,  
$$K'=(\{u\}\times \{y\})\cup (\{x_{i_2}\}\times F'_{i_2})\cup  \dots \cup (\{x_{i_k}\}\times F'_{i_k}) ,$$
and  $K''$ is a facet of $G[\mathcal{H}]$ associated to $F''$ containing $\{x_{i_r}\}\times F'_{i_r}$ for every $2\leq r \leq k$. It is easily seen that $K''<K', (u,y)\in K'\setminus K$ and $K'\setminus K''=\{(u,y)\}$ for some $y\in H_u$.
\end{proof}

In the following examples, by Theorem \ref{shell2}, $G[\mathcal{H}]$ is shellable.
\begin{examples}\label{example}
\begin{itemize}
\item[a.] \cite[Theorem 3.12]{SF} $H_x$ is complete for all $x\in V(G)$ and $G$ is shellable. 
\item[b.] $G$ is a complete graph with $n$ vertices and $\mathcal{H}$ is a family of shellable graphs with $(n-1)$ complete graphs.
\item[c.] $G$ is a graph with $V(G)=\{a,b,c,d,e\}$ and $E(G)=\{ab, bc, cd\}$ and $\mathcal{H}=\{H_a, H_b, H_c, H_d, H_e\}$ in which $H_a, H_c$ and $H_e$ are arbitrary shellable graphs and $H_b$ and $H_d$ are some complete graphs. It can be easily seen that $\{a,c,e\}<\{a,d,e\}<\{b,d,e\}$ is a shelling of $G$ satisfying the hypothesis of Theorem \ref{shell2}. 
\end{itemize}
\end{examples}

The following corollary gives us a new class of Cohen-Macaulay graphs.
\begin{corollary}\label{cor2.16}
Suppose that $G$ is a complete graph with $n$ vertices and  $\mathcal{H}$ is a family of Cohen-Macaulay shellable graphs with $(n-1)$ complete graphs and for each $x,y \in V(G)$ we have $\alpha(H_x)=\alpha(H_y)$. Then $G[\mathcal{H}]$ is Cohen-Macaulay.
\end{corollary}
\begin{proof}
By \cite[Corollary 5.3.12]{V2}, $H_x$s are unmixed. Since $G$ is complete, its maximal independent sets of vertices are singleton. So it can be deduced from \cite[Theorem 2]{T+V} that $G[\mathcal{H}]$ is unmixed.  On the other hand, by Examples \ref{example}(b), $G[\mathcal{H}]$ is shellable. Hence \cite[Theorem 5.3.18]{V2} completes the proof.
\end{proof}

Henceforth we are going to study the vertex decomposability and shedding vertices of $G[\mathcal{H}]$. We begin by the following proposition about shedding vertices of $G[\mathcal{H}]$.

%\begin{lemma}\label{2.13}
%If $x\in V(G)$ and $H_x$ is isomorphic to $O_n$ for some $n>1$, then $(x,y)$ is not a shedding vertex of $G[\mathcal{H}]$ for all $y\in H_x$. 

%In particular if $G$ is not totally disconnected and $H_x$s are totally disconnected with more than one vertex for all $x\in V(G)$, then $G[\mathcal{H}]$ is not vertex decomposable.
%\end{lemma}
%\begin{proof}
%For each $y\in H_x$ if $x$ is an isolated vertex of $G$, then $G[\mathcal{H}]\setminus (x,y)=G[\mathcal{H}]\setminus N_{G[\mathcal{H}]}[(x,y)]$ and so the result is clear. Suppose $V(H_x)=\{y_1, \dots, y_n\}, N_G(x)=\{x_1, \dots , x_k\}$ and in contrary $(x, y_1)$ is a shedding vertex of $G[\mathcal{H}]$. Then since $H_x$ is totally disconnected, for each $1\leq j \leq n$,  
%$$N_{G[\mathcal{H}]}((x,y_j))=\bigcup_{i=1}^k\left( \{x_i\}\times H_{x_i}\right) .$$
 %So, every maximal independent set of vertices $S$ in $G[\mathcal{H}]\setminus N_{G[\mathcal{H}]}[(x,y_1)]$ should contain $\{(x,y_2), \dots , (x,y_n)\}$. Now since $n\geq 2$ and  for each $2\leq j \leq n$ we have $N_{G[\mathcal{H}]}((x,y_j))=N_{G[\mathcal{H}]}((x,y_1))$, $S$ must be a maximal one in $G[\mathcal{H}]\setminus (x,y_1)$ which is a contradiction.
%\end{proof}

%Note that Lemma \ref{2.13} shows that vertex decomposability of $G$ and $H_x$s not necessarily implies that $G[\mathcal{H}]$ is vertex decomposable. As a corollary one may conclude that every complete $n$-partite graph $(n\geq 2)$ whose each partition has more than one vertex is not vertex decomposable.

\begin{proposition}\label{lemma2.10}
Assume that $x\in V(G)$ and $y\in V(H_x)$. 
\begin{itemize}
\item[a.] If
\begin{itemize}
\item[1.] $H_x$ is $C_5$-free and $y$ is a shedding vertex of $H_x$, or
\item[2.] $x$ and $y$ are some  shedding vertices of $G$ and $H_x$ respectively,
\end{itemize}
then $(x,y)$ is a shedding vertex of $G[\mathcal{H}]$. 
\item[b.] If $(x,y)$ is a shedding vertex of $G[\mathcal{H}]$, then either $H_x=\{y\}$ or $y$ is a shedding vertex of $H_x$.
\end{itemize}
\end{proposition}
\begin{proof}
a.1. Suppose that $x\in V(G)$, $H_x$ is $C_5$-free and $y$ is a shedding vertex of $H_x$. In view of \cite[Lemma 2.3]{FS}, there is a vertex $y'$ in $N_{H_x}(y)$ such that $N_{H_x}[y']\subseteq N_{H_x}[y]$. Note that 
$$N_{G[\mathcal{H}]}((x,y))=\left(\bigcup_{x'\in N_G(x)}(\{x'\}\times H_{x'})\right) \cup (\{x\}\times N_{H_x}(y)).$$
Assume that $S$ is an independent set of vertices in $G[\mathcal{H}] \setminus N_{G[\mathcal{H}]}[(x,y)]$. It is easily seen that $(x,y')$ is a neighbor of $(x,y)$ such that  $N_{G[\mathcal{H}]}((x,y'))$ is contained in $N_{G[\mathcal{H}]}[(x,y)]$. Hence $S\cup \{(x,y')\}$ is independent in $G[\mathcal{H}]\setminus (x,y)$. This shows that $(x,y)$ is a shedding vertex of $G[\mathcal{H}]$ as desired.

a.2. Suppose that  $x$ is a shedding vertex of $G$, $y$ is a shedding vertex of $H_x$ and $S$ is a maximal independent set of vertices in $G[\mathcal{H}] \setminus N_{G[\mathcal{H}]}[(x,y)]$. If $N_{H_x}[y]\neq H_x$, then $S'=\{y'\in H_x \ | \ (x,y')\in S\}\neq \emptyset$. Since $S$ is independent in $G[\mathcal{H}] \setminus N_{G[\mathcal{H}]}[(x,y)]$,  $S'$ will be independent in $H_x\setminus N_{H_x}[y]$. Hence there is a vertex $y_0\in N_{H_x}(y)$ such that $S'\cup \{y_0\}$ is independent. This shows that $(x,y_0)\in N_{G[\mathcal{H}]}((x,y))$ and $S\cup \{(x,y_0)\}$ is independent in $G[\mathcal{H}]\setminus (x,y)$. Otherwise assume that $N_{H_x}[y]= H_x$. Then if $N_G[x]=G$, then $S=\emptyset$ and the result clearly holds. Else $$S'=\{x'\in G \ | \ x'\neq x, (x',y')\in S \text{ for some } y'\in H_{x'}\}\neq \emptyset .$$
Therefore $S'$ is an independent set of vertices in $G\setminus N_G[x]$. So there is a vertex $x_0 \in N_G(x)$ such that $S'\cup \{x_0\}$ is independent. Hence $(x_0,y_0)\in N_{G[\mathcal{H}]}((x,y))$ for each $y_0\in H_{x_0}$ and $S\cup \{(x_0,y_0)\}$ is independent in $G[\mathcal{H}]\setminus (x,y)$. These show that $(x,y)$ is a shedding vertex of $G[\mathcal{H}]$.

b. Suppose that $(x,y)$ is a shedding vertex of $G[\mathcal{H}]$ and $S'$ is a maximal independent set of vertices in $H_x\setminus N_{H_x}[y]$. If $S'=\emptyset$, the result obviously holds. Else set $S=\{(x,y') \ | \ y'\in S'\}$. It is clear that $S$ is independent in $G[\mathcal{H}]\setminus N_{G[\mathcal{H}]}[(x,y)]$. Hence there exist a vertex $(x_0,y_0)\in N_{G[\mathcal{H}]}((x,y))$ such that $S\cup \{(x_0, y_0)\}$ is independent in $G[\mathcal{H}]\setminus (x,y)$. Since every vertex in $\bigcup_{x'\in N_G(x)}(\{x'\}\times H_{x'})$ is adjacent to every vertex in $S$, we should have $x_0=x$ and $y_0\in N_{H_x}(y)$. So $S'\cup \{y_0\}$ should be independent as desired. This shows that $y$ is a shedding vertex of $H_x$.
\end{proof}

%\begin{lemma}\label{lem}
%If $y$ is a shedding vertex of a non-complete graph $H$, then $H\setminus y$ is also non-complete.
%\end{lemma}
%\begin{proof}
%In contrary assume that $H\setminus y$ is a complete graph. Then two statements occur:

%If $N_H[y]=H$, then since $H\setminus y$ is complete, $H$ should be complete which is a contradiction.

%If $N_H[y]\neq H$, then there exists a vertex $y'\in H\setminus N_H[y]$. Since $H\setminus y$ is complete, $\{y'\}$ is a maximal independent set of vertices in both of $H\setminus y$ and $H\setminus N_H[y]$ which is a contradiction.
%\end{proof}

Although by Proposition \ref{lemma2.10} containing a shedding vertex is feasible for $G[\mathcal{H}]$, but there are examples of vertex decomposable graphs $G$ and $H$ with shedding vertices such that $G[H]$ is not vertex decomposable. For instance if $x$ is a shedding vertex of $P_3$, then $(P_3[P_3])\setminus (x,x)$ is not vertex decomposable, while $(x,x)$ is a shedding vertex of $P_3[P_3]$ by Proposition \ref{lemma2.10}(a2). Also, one can easily see that $K_2[\{K_1, P_3\}]$ is vertex decomposable. These examples illustrate that stating the following theorem, which is a generalization of vertex decomposable part of \cite[Theorem 1.1]{Vantuyl}, is not without grace.

\begin{theorem}\label{2.15}
\begin{itemize}
\item[1.] If $H_x$ is complete for every $x\in V(G)$, then $G[\mathcal{H}]$ is vertex decomposable if and only if $G$ is vertex decomposable. 
\item[2.] If $G$ is totally disconnected, then $H_x$ is vertex decomposable for every $x\in V(G)$ if and only if $G[\mathcal{H}]$ is vertex decomposable.
\item[3.] If $G$ is not totally disconnected and $G[\mathcal{H}]$ is a vertex decomposable graph, then $H_x$ is complete for some $x\in V(G)$.
\item[4.] If $G[\mathcal{H}]$ is vertex decomposable, then $G$ and $H_x$ are also vertex decomposable for all $x\in V(G)$.
\end{itemize}
\end{theorem}
\begin{proof}
1. follows from \cite[Theorem 3.7]{SF}.

2. If $G$ is a totally disconnected graph, then $G[\mathcal{H}]$ is disjoint union of $H_x$s. So, in view of Lemma \ref{lemma2.4}, $H_x$ is vertex decomposable for every $x\in V(G)$ if and only if $G[\mathcal{H}]$ is vertex decomposable.

3.  immediately follows from Theorem \ref{shell}(3), since every vertex decomposable graph is shellable.

%, but here we also bring a straight proof for it as follows:

%Suppose that $H_x$ is not  complete for all $x\in V(G)$. Then we show that $G[\mathcal{H}]$ cannot be vertex decomposable. We proceed by induction on $\sum_{x\in V(G)}|H_x|=n$. Since $G$ is not totally disconnected and $H_x$s are not complete, $n\geq 4$. If $n=4$, then by Lemma \ref{2.13} the result holds. Assume that $n\geq 5$ and  the result holds for smaller values of $n$. Suppose in contrary that $G[\mathcal{H}]$ is vertex decomposable. Then it has a shedding vertex, say $(x_0,y_0)$ such that $G[\mathcal{H}]\setminus (x_0,y_0)$ is vertex decomposable. In view of Lemma \ref{lemma2.10}(b), $y_0$ is a shedding vertex of $H_{x_0}$.  We know that 
%$$G[\mathcal{H}]\setminus (x_0,y_0)=G[\mathcal{H}'],$$
%where $\mathcal{H}'=\{H_x | x\in V(G), x\neq x_0\}\cup \{H_{x_0}\setminus y_0\}$. Since $y_0$ is a shedding vertex of $H_{x_0}$ and $H_{x_0}$ is not complete, $H_{x_0}\setminus y_0$ is not complete by Lemma \ref{lem}. Therefore inductive hypothesis implies that $G[\mathcal{H}']$ and so $G[\mathcal{H}]\setminus (x_0,y_0)$ cannot be vertex decomposable which is a contradiction.

4. If $|V(G[\mathcal{H}])|\leq 2$, then the result clearly holds. Suppose inductively that $|V(G[\mathcal{H}])|>2$ and the result has been proved for smaller values of $|V(G[\mathcal{H}])|$. If $G[\mathcal{H}]$ is totally disconnected, then so are $G$ and $H_x$s. Else assume that  $V(G)=\{x_1, \dots , x_n\}$ and $H_i=H_{x_i}$. Without loss of generality suppose that $(x_1,y)$ is a shedding vertex of $G[\mathcal{H}]$ for some $y\in V(H_1)$ such that $G[\mathcal{H}]\setminus (x_1,y)$ and $G[\mathcal{H}]\setminus N_{G[\mathcal{H}]}[(x_1,y)]$ are vertex decomposable. It can be checked that
$$G[\mathcal{H}]\setminus (x_1,y)=G[\mathcal{H}'],$$
$$G[\mathcal{H}]\setminus N_{G[\mathcal{H}]}[(x_1,y)]=\{x_1\}[H_1\setminus N_{H_1}[y]]\cup (G\setminus N_G[x_1])[\mathcal{H}''],$$
where  $\mathcal{H}'=\{H_1\setminus y, H_2, \dots , H_n\}, G\setminus N_G[x_1]=\{x_{j_1}, \dots , x_{j_k}\}$ and $\mathcal{H}''=\{H_{j_1}, \dots , H_{j_k}\}$. Note that if $H_1=\{y\}$, then
 $$G[\mathcal{H}]\setminus (x_1,y)=(G\setminus x_1)[\{H_2, \dots , H_n\}].$$
  So we prove the result in each case:

Case I. $H_1\neq \{y\}$. In this case, vertex decomposability of $G[\mathcal{H}]\setminus (x_1,y)$ together with inductive hypothesis implies that $G$ and $H_i$ are vertex decomposable for all $2\leq i \leq m$ and $H_1\setminus y$ is vertex decomposable. Now if $H_1= N_{H_1}[y]$, then $H_1$ is also vertex decomposable. Otherwise, vertex decomposability of $G[\mathcal{H}]\setminus N_{G[\mathcal{H}]}[(x_1,y)]$ and Lemma \ref{lemma2.4} imply that $H_1\setminus N_{H_1}[y]$ is vertex decomposable. It remains to prove that $y$ is a shedding vertex of $H_1$. To this aim suppose that $S$ is a non-empty maximal independent set of vertices in $H_1\setminus N_{H_1}[y]$. Then $S'=\{(x_1,s) \ | \ s\in S\}$ is a non-empty independent set of vertices in $G[\mathcal{H}]\setminus N_{G[\mathcal{H}]}[(x_1,y)]$. So there is a vertex $(x_j,y')\in N_{G[\mathcal{H}]}[(x_1,y)]$ such that $S'\cup \{(x_j,y')\}$ is independent in $G[\mathcal{H}]\setminus (x_1,y)$. If $x_j\in N_G(x_1)$, then $(x_j,y')$ is adjacent to each vertex of $S'$ and so this is a contradiction. Hence $j=1, y'\in N_H(y)$ and $S\cup \{y'\}$ is independent in $H_1$. These show that $H_1$ is also vertex decomposable.

Case II. $H_1=\{y\}$. Vertex decomposability of $G[\mathcal{H}]\setminus (x_1,y)$ and $G[\mathcal{H}]\setminus N_{G[\mathcal{H}]}[(x_1,y)]$, Lemma \ref{lemma2.4} and inductive hypothesis establish that  $G\setminus x_1, G\setminus N_G[x_1]$ and $H_i$s are vertex decomposable. Now suppose that $S=\{x_{\ell_1}, \dots , x_{\ell_r}\}$ is a maximal independent set of vertices in $G\setminus N_G[x_1]$ . Then for each $1\leq i \leq r$ choose one vertex $y_{\ell_i}$ in $H_{\ell_i}$ and set $S'=\{(x_{\ell_i}, y_{\ell_i}) \ | \ 1\leq i \leq r\}$. It is clear that $S'$ is independent in $G[\mathcal{H}]\setminus N_{G[\mathcal{H}]}[(x_1,y)]$. So there is a vertex $(x_j,y')\in N_{G[\mathcal{H}]}[(x_1,y)]$ such that $S'\cup \{(x_j,y')\}$ is independent. Since $H_1=\{y\}$, we should have $x_j\in N_G(x_1)$ and $S\cup \{x_j\}$ is independent. This shows that $G$ is also vertex decomposable.
\end{proof}
%We end our paper by the following corollary which immediately follows from Theorem \ref{2.15}.
%\begin{corollary}(Compare \cite[Theorem 1.1]{Vantuyl}.)
%\begin{itemize}
%\item[a.] If $G$ is a totally disconnected graph, then $H$ is vertex decomposable if and only if $G[H]$ is vertex decomposable.
%\item[b.] If $G$ is not totally disconnected, then $G[H]$ is vertex decomposable if and only if $G$ is vertex decomposable and $H$ is a complete graph.
%\end{itemize}
%\end{corollary}

\begin{corollary}
If $G[\mathcal{H}]$ is a vertex decomposable and Cohen-Macaulay graph, then for each $x\in V(G)$, $H_x$ is Cohen-Macaulay.
\end{corollary}
\begin{proof}
By Theorem \ref{2.15}(4), $H_x$s are vertex decomposable. On the other hand \cite[Corollary 5.3.12]{V2} and \cite[Theorem 2]{T+V} show that $H_x$s are unmixed. Now, the result follows from \cite[Theorem 11.3]{BW2} and \cite[Theorem 5.3.18]{V2}.
\end{proof}

\textit{Acknowledgements.}
The author would like to thank the referee for interest in the
subject, careful reading and comments on this article.

\end{document}